\documentclass[12pt,hyp,]{amsart}

\usepackage{rotating}
\usepackage{mathrsfs}  
\usepackage{tikz}
\usepackage{tikz-cd}
\usetikzlibrary{knots}
\usepackage{subfig}
\usepackage{adjustbox}
\usepackage{cite}
\usepackage[colorlinks, linkcolor=blue, allcolors=blue]{hyperref}
\hypersetup{nesting=true,debug=true,naturalnames=true}
\usepackage[margin=1in]{geometry}

\usetikzlibrary{decorations.markings}

\newcommand{\Z}{{\mathbb{Z}}}

\theoremstyle{definition}
\newtheorem{definition}{Definition}[section]
\newtheorem{lemma}[definition]{Lemma}

\newtheorem{corollary}[definition]{Corollary}
\newtheorem{remark}[definition]{Remark}
\newtheorem{theorem}[definition]{Theorem}
\newtheorem{conjecture}[definition]{Conjecture}

\newtheorem{example}[definition]{Example}

\title{Odd order group actions on alternating knots} 

\author{Keegan Boyle}  

\email{kboyle@uoregon.edu}

\begin{document}
\begin{abstract}
Let $K$ be a an alternating prime knot in $S^3$. We investigate the category of flypes between reduced alternating diagrams for $K$. As a consequence, we show that any odd prime order action on $K$ is isotopic through maps of pairs to a single flype. This implies that for any odd prime order action on $K$ there is either a reduced alternating periodic diagram or a reduced alternating free periodic diagram. Finally, we deduce that the quotient of an odd periodic alternating knot is also alternating.
\end{abstract}
\maketitle
\section{Introduction}
A diagram $D$ for a knot $K \subset S^3$ is \emph{alternating} if the crossings in $D$ alternate between under and over crossings as you follow $K$ around the diagram, and a knot $K$ is \emph{alternating} if it has an alternating diagram. See Figure \ref{fig:4} for an example. Recently, Greene \cite{G} and Howie \cite{H} each showed that an alternating knot can instead be characterized by the existence of certain spanning surfaces. In light of this more geometric interpretation, it is interesting to consider how the property of being alternating interacts with finite order group actions on $K$. Specifically, we have the following conjecture.
\begin{conjecture}
\label{conj:main}
The quotient of an alternating periodic knot $K \subset S^3$ is alternating. 
\end{conjecture}
To approach this conjecture, we use a theorem of Menasco and Thistlethwaite \cite{MT} that any homeomorphism $f:(S^3,K) \to (S^3,K)$, can be realized up to isotopy through maps of pairs by a sequence of certain diagrammatic moves called flypes. In particular, we study the category of flypes on an alternating knot, and as a consequence classify odd prime order group actions on alternating knots. Specifically, we prove the following main result. 
\begin{theorem}
\label{thm:single}
Let $\tau$ be an odd prime order $p$ action on a prime alternating knot $K$. Then there exists a reduced alternating diagram $D$ for $K$, and a flype $f$ from $D$ to $D$ such that $\tau$ is isotopic through maps of pairs $(S^3,K)$ to $f$. 
\end{theorem}
We then obtain the following consequences for periodic and free periodic actions, proving Conjecture \ref{conj:main} when the order of the period is odd. 
\begin{corollary}
\label{cor:alt}
If $K$ is a $p$ periodic prime alternating knot for an odd prime $p$, then $K$ has a reduced alternating periodic diagram. 
\end{corollary}
\begin{corollary}
\label{cor:free}
If $K$ is an odd prime $p$ free periodic alternating hyperbolic knot, then $K$ has a reduced alternating free periodic diagram. See Figure \ref{fig:free}.
\end{corollary}
While preparing this paper, the author discovered that Costa and Hongler released \cite{CH}, which contains overlapping results. In particular, Corollary \ref{cor:alt}, one of the main goals of this paper, also appears there. However, Corollary \ref{cor:free} does not appear in \cite{CH}, while their paper also discusses the case of certain 2 periodic actions. Both this paper and \cite{CH} use flypes as a main tool, but differ in their techniques.

Throughout this paper, all knots are prime, alternating, and contained in $S^3$.
\subsection{Organization}
Section \ref{sec:flype} defines the relevant notions of flypes and their equivalences, Section \ref{sec:main} proves the main results, and Section \ref{sec:example} gives a few example applications.
\subsection{Acknowledgments}
The author would like to thank Liam Watson and Joshua Greene for helpful conversations, and Robert Lipshitz for his support and belief in this project.
\section{The Category of Flypes \label{sec:flype}}
In this section we define the category of flypes for a given prime alternating knot $K \subset S^3$, which has objects roughly corresponding to diagrams for $K$, and morphisms generated by flypes. 
\begin{definition}
The \emph{standard crossing ball} $B_{\mbox{\tiny{std}}} = (B^3, D^2, a_1 \cup a_2)$ is the triple of 
\begin{enumerate}
\item the 3-ball $\{(x,y,z) \in \mathbb{R}^3 \mid |(x,y,z)| \leq 1\}$,
\item the horizontal unit disk inside this ball $\{(x,y,z) \in \mathbb{R}^3 \mid z=0 \mbox{ and } |(x,y,z)| \leq 1\}$, and
\item the union of the two arcs $a_1 = \{(x,y,z) \in \mathbb{R}^3 \mid x=0, z \geq 0 \mbox{ and } y^2 + z^2 = 1\}$ and $a_2 = \{(x,y,z) \in \mathbb{R}^3 \mid y=0, z\leq 0 \mbox{ and } x^2 + z^2 = 1\}$.
\end{enumerate}
\end{definition}
\begin{definition}
A \emph{realized diagram} $\lambda(D)$ for a knot $K \subset S^3$ is smooth embeddings
\begin{enumerate}
\item $S^2 \hookrightarrow S^3$, the \emph{projection sphere},
\item $K: S^1 \hookrightarrow S^3$, the \emph{knot}, and
\item $\{B_i\} \hookrightarrow S^3$, the \emph{crossing balls},
\end{enumerate}
such that the $\{B_i\}$ are disjoint and $K \subset S^2 \cup \{B_i\}$, along with homeomorphisms of triples $c_i: (B_i, B_i \cap S^2, B_i \cap K)\to B_{\mbox{\tiny{std}}}$, the \emph{crossing ball identification maps}. The \emph{diagram} $D$ is the labeled graph in $S^2$ which is the projection of $K$ with vertices labeled to reflect under and over crossings.
\end{definition}
\begin{definition}
An \emph{isomorphism} of realized diagrams $f: \lambda(D) \to \lambda(\overline{D})$ is a homeomorphism of pairs $f: (S^3, K) \to (\overline{S}^3, \overline{K})$ such that $f(S^2)$ is isotopic to $\overline{S}^2$ relative to $\overline{K}$, $f(B_i) = \overline{B}_i$ and $\overline{c}_i \circ f= c_i$. 
\end{definition}
It is immediate that if $\lambda(D)$ and $\lambda'(D)$ are realized diagrams for isomorphic labeled graphs $D$, then there is an isomorphism of realized diagrams $f: \lambda(D) \to \lambda'(D)$, and vice versa. 

\begin{definition}
A \emph{standard flype} is a transformation between realized diagrams $\lambda(D) \to \lambda(E)$ of the form shown in Figure \ref{fig:flype}, where both tangles $T_1$ and $T_2$ are required to be non-trivial. That is, a homeomorphism $f:S^3 \to S^3$ which restricts to the identity on a round ball containing $T_2$ (shown as the exterior of $\alpha_2$), $\pi$ rotation around the horizontal axis on a ball containing $T_1$ (shown as the interior of $\alpha_1$) and a linear homotopy between them to get a homeomorphism. We further fix once and for all a homeomorphism $c_{\mbox{\tiny{std}}}: c_1 \to B_{\mbox{\tiny{std}}}$, and require that this be the crossing ball identification map used in $\lambda(D)$ and its 180 degree rotation be the crossing ball identification map used in $\lambda(E)$.
\end{definition}

\begin{definition}
A \emph{flype} $f: \lambda(D) \to \lambda(D')$ is any composition $f = g_1 \circ s \circ g_2$ where $g_1$ and $g_2$ are isomorphisms of realized diagrams, and $s$ is a standard flype. We will refer to the crossing ball in $\lambda(D')$ created by $f$ as $c_f$, and the crossing ball in $\lambda(D)$ removed by $f$ as $c^f$. The ball $\alpha_1$ containing the tangle $T_1$ will be referred to as the \emph{domain} of $f$. We also consider an isomorphism of realized diagrams to be a flype, and refer to it as the \emph{trivial} flype. 
\end{definition}
Now consider a flype $f: \lambda(D) \to \lambda(D')$. Then in the planar graph projection of $D$ we get a distinguished crossing $c^f$ and a distinguished pair of edges $(e^1_f, e^2_f)$ which will cross to form $c_f$. The following lemma states that this is enough to reconstruct the flype. 
\begin{lemma}
\label{lemma:unique}
Let $f: \lambda(D) \to \lambda(D')$ and $g: \lambda'(D) \to \lambda'(D')$ be flypes such that $(e^1_f, e^2_f) = (e_1^g,e_2^g)$ and $c^f = c^g$. Then there exists a pair of isomorphisms of realized diagrams $g_1,g_2$ such that $f = g_1 \circ g \circ g_2$.
\end{lemma}
\begin{proof}
To begin, note that there is an isomorphism between $\lambda(D)$ and $\lambda'(D)$, so that we may consider $f$ and $g$ to start at the same realized diagram. Similarly, there is an isomorphism from $\lambda(D')$ to $\lambda'(D')$, so we may assume $f$ and $g$ end at the same realized diagram. Now note that $f$ and $g$ induce maps on underlying graphs in $S^2$ which are homotopic relative to the vertices of the graph. In particular, $f$ and $g$ restrict to the same map on crossing balls since both are determined by the crossing ball identification maps for $\lambda(D)$ and $\lambda(D')$. From there we have a unique extension to the rest of $S^3$ up to an isomorphism of realized diagrams, as desired.
\end{proof}

\begin{figure}
\scalebox{0.7}{
\begin{tikzpicture}
\begin{knot}[
%draft mode=crossings, 
clip width = 15, 
flip crossing = 1,
%flip crossing = 2
]
\strand[black , thick] (8,0) .. controls +(1,0) and +(0,1) .. (9,-1) .. controls +(0,-2) and +(-3,-3) .. (0,0) .. controls +(0,0) and +(-1,0) .. (2,1);
\strand[black , thick] (8,1) .. controls +(1,0) and +(0,-1) .. (9,2) .. controls +(0,2) and +(-3,3) .. (0,1) .. controls +(0,0) and +(-1,0) .. (2,0);
\end{knot}
\draw[black, ultra thick] (2, -.53) -- (2,1.5) -- (4,1.5) -- (4,-.5) -- (2,-.5);
\draw[black, thick] (4,1) -- (6,1);
\draw[black, thick] (4,0) -- (6,0);
\draw[black, ultra thick] (6, -.53) -- (6,1.5) -- (8,1.5) -- (8,-.5) -- (6,-.5);

\draw[black, thick, dotted] (1,.5) .. controls +(0,2.5) and +(0,2.5) .. (4.5,.5);
\draw[black, thick, dotted] (4.5,.5) .. controls +(0,-2.5) and +(0,-2.5) .. (1,.5);

\draw[black, thick, dotted] (0,.5) .. controls +(0,3.5) and +(0,3.5) .. (5.5,.5);
\draw[black, thick, dotted] (5.5,.5) .. controls +(0,-3.5) and +(0,-3.5) .. (0,.5);

\node at (3,.5) {\huge{$T_1$}};
\node at (7,.5) {\huge{$T_2$}};

\node at (4.2, 2){$\alpha_1$};
\node at (4.3, -2){$\alpha_2$};
\node at (0.65,0.8){$c_1$};
\node at (4,-4){\huge{$\lambda(D)$}};
\end{tikzpicture}}
\scalebox{0.7}{
\begin{tikzpicture}
\begin{knot}[
%draft mode=crossings, 
clip width = 15, 
%flip crossing = 1,
%flip crossing = 2
]
\strand[black, thick] (4,1) .. controls +(1,0) and +(-1,0) .. (6,0);
\strand[black, thick] (4,0) .. controls +(1,0) and +(-1,0) .. (6,1);
\end{knot}
\draw[black, ultra thick] (2, -.53) -- (2,1.5) -- (4,1.5) -- (4,-.5) -- (2,-.5);
\draw[black, thick, dotted] (1,.5) .. controls +(0,2.5) and +(0,2.5) .. (4.5,.5);
\draw[black, thick, dotted] (4.5,.5) .. controls +(0,-2.5) and +(0,-2.5) .. (1,.5);

\draw[black, thick, dotted] (0,.5) .. controls +(0,3.5) and +(0,3.5) .. (5.5,.5);
\draw[black, thick, dotted] (5.5,.5) .. controls +(0,-3.5) and +(0,-3.5) .. (0,.5);
\draw[black, ultra thick] (6, -.53) -- (6,1.5) -- (8,1.5) -- (8,-.5) -- (6,-.5);

\draw[black , thick] (8,0) .. controls +(1,0) and +(0,1) .. (9,-1) .. controls +(0,-2) and +(-2,-2) .. (0,-1) .. controls +(1,1) and +(-1,0) .. (2,0);
\draw[black , thick] (8,1) .. controls +(1,0) and +(0,-1) .. (9,2) .. controls +(0,2) and +(-2,2) .. (0,2) .. controls +(1,-1) and +(-1,0) .. (2,1);

\node at (3,.5) {\raisebox{\depth}{\scalebox{1}[-1]{\huge{$T_1$}}}};
\node at (7,.5) {\huge{$T_2$}};

\node at (4.2, 2){$\alpha_1$};
\node at (4.3, -2){$\alpha_2$};
\node at (5,0.9){$c_2$};
\node at (4,-4){\huge{$\lambda(E)$}};
\end{tikzpicture}}
\caption{A standard flype fixes the exterior of $\alpha_2$ and reflects the interior of $\alpha_1$ across the horizontal axis, with a linear homotopy in between. It removes the crossing ball at $c_1$ and creates a crossing ball at $c_2$.}
\label{fig:flype}
\end{figure}
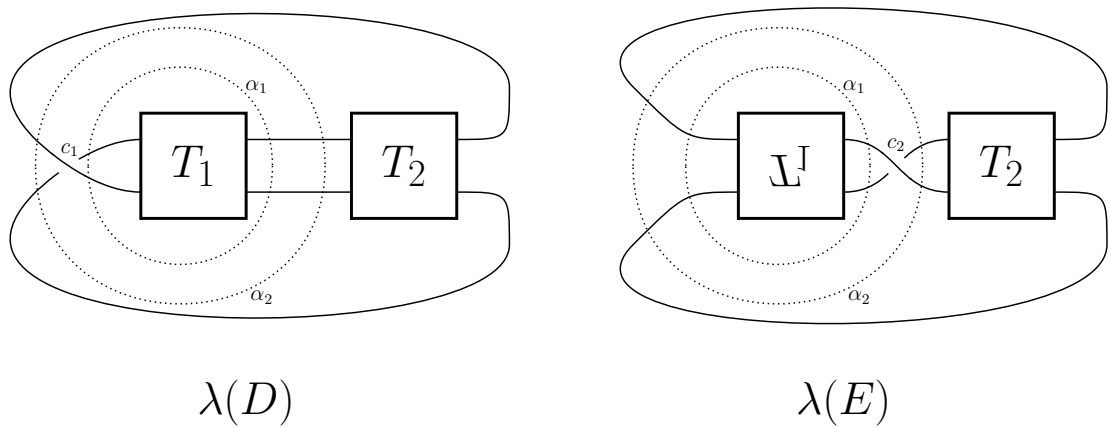
\begin{lemma}
\label{lemma:exists}
Let $\lambda(D)$ and $\lambda(D')$ be realized reduced alternating diagrams for $K$, and let $N(K)$ be a neighborhood of $K$. Then if $f: \lambda(D) \to \lambda(D')$ is a homeomorphism $S^3 \to S^3$ such that $f$ agrees with a flype when restricted to $(N(K)\cap S^2)\cup\{c_i\}$, then $f$ is a flype. That is, if $f$ restricts to a flype on the underlying diagram, then $f$ is a flype.
\end{lemma}
\begin{proof}
Let $\varphi$ be a flype from $\lambda(D) \to \lambda(D')$, so that $f$ and $\varphi$ agree when restricted to both the crossing balls and a neighborhood of $K$ in $S^2$. Now let $g = f \circ \varphi^{-1}:\lambda(D') \to \lambda(D')$, and observe that $g$ is the identity map on each crossing ball, and on $N(K) \cap S^2$. In particular, this determines the relative isotopy class of $g(S^2)'$ so that $g$ is an isomorphism of realized diagram. But then $f = g \circ \varphi$, so $f$ is a flype, as desired.
\end{proof}
By combining Lemmas \ref{lemma:unique} and \ref{lemma:exists}, we see that a diagrammatic description of a flype is sufficient since it corresponds to a unique flype up to isomorphism of realized diagrams.

Now given a composition of two flypes, the following lemmas will allow us to either combine them into a single flype, or else (roughly) commute them past each other. 
\begin{lemma}
\label{lemma:combine}
If $f_1: \lambda(D_1) \to \lambda(D_2)$ and $f_2: \lambda(D_2) \to \lambda(D_3)$ are flypes such that the crossing created by $f_1$ is the crossing removed by $f_2$, and $K$ is prime, then there exists a flype $f_{1,2}: \lambda(D_1) \to \lambda(D_3)$ with $f_2 \circ f_1 = f_{1,2}$.
\end{lemma}
\begin{proof}
By Lemmas \ref{lemma:unique} and \ref{lemma:exists}, it is enough to consider diagrammatic flypes. We first note that there are three possible configurations for $f_2$ relative to $f_1$. See Figure \ref{fig:orientations}. Observe, however, that configuration (A) is impossible. Indeed, if $T_3$ or $T_4$ is a non-trivial tangle, then $K$ cannot be prime. On the other hand, in configuration (C), the composition will flip the tangle $T_1$ over twice so that the composition is simply a flype with domain $T_2$. Similarly, in configuration (B), the composition will flip both tangles $T_1$ and $T_2$ over once, which can be realized as a single flype on their sum. 
\end{proof}
\begin{figure}
\scalebox{0.7}{
\begin{tikzpicture}
\begin{knot}[
%draft mode=crossings, 
clip width = 15, 
flip crossing = 1,
%flip crossing = 2
]
\strand[black , thick] (8,0) .. controls +(1,0) and +(0,1) .. (9,-1) .. controls +(0,-2) and +(-3,-3) .. (0,0) .. controls +(0,0) and +(-1,0) .. (2,1);
\strand[black , thick] (8,1) .. controls +(1,0) and +(0,-1) .. (9,2) .. controls +(0,2) and +(-3,3) .. (0,1) .. controls +(0,0) and +(-1,0) .. (2,0);
\end{knot}
\draw[black, ultra thick] (2, -.53) -- (2,1.5) -- (4,1.5) -- (4,-.5) -- (2,-.5);
\draw[black, thick] (4,1) -- (6,1);
\draw[black, thick] (4,0) -- (6,0);
\draw[black, ultra thick] (6, -.53) -- (6,1.5) -- (8,1.5) -- (8,-.5) -- (6,-.5);

\draw[black, thick, dotted] (5,.5) .. controls +(0,2.5) and +(0,2.5) .. (8.5,.5);
\draw[black, thick, dotted] (8.5,.5) .. controls +(0,-2.5) and +(0,-2.5) .. (5,.5);

\draw[black, thick, dotted] (.5,.5) .. controls +(0,3) and +(0,3) .. (5,.5);
\draw[black, thick, dotted] (5,.5) .. controls +(0,-3) and +(0,-3) .. (.5,.5);

\node at (3,.5) {\huge{$T_2$}};
\node at (7,.5) {\huge{$T_1$}};

\node at (4, 2){$\alpha_{f_1}$};
\node at (6.3, -1.7){$\alpha_{f_2}$};
\node at (4,-4){\huge{$(B)$}};
\end{tikzpicture}} \hskip -.5 in.
\scalebox{0.7}{
\begin{tikzpicture}
\begin{knot}[
%draft mode=crossings, 
clip width = 15, 
flip crossing = 1,
%flip crossing = 2
]
\strand[black , thick] (8,0) .. controls +(1,0) and +(0,1) .. (9,-1) .. controls +(0,-2) and +(-3,-3) .. (0,0) .. controls +(0,0) and +(-1,0) .. (2,1);
\strand[black , thick] (8,1) .. controls +(1,0) and +(0,-1) .. (9,2) .. controls +(0,2) and +(-3,3) .. (0,1) .. controls +(0,0) and +(-1,0) .. (2,0);
\end{knot}
\draw[black, ultra thick] (2, -.53) -- (2,1.5) -- (4,1.5) -- (4,-.5) -- (2,-.5);
\draw[black, thick] (4,1) -- (6,1);
\draw[black, thick] (4,0) -- (6,0);
\draw[black, ultra thick] (6, -.53) -- (6,1.5) -- (8,1.5) -- (8,-.5) -- (6,-.5);

\draw[black, thick, dotted] (5,.5) .. controls +(0,2.5) and +(0,2.5) .. (8.5,.5);
\draw[black, thick, dotted] (8.5,.5) .. controls +(0,-2.5) and +(0,-2.5) .. (5,.5);

\draw[black, thick, dotted] (.5,.5) .. controls +(0,3.5) and +(0,3.5) .. (8.5,.5);
\draw[black, thick, dotted] (8.5,.5) .. controls +(0,-3.5) and +(0,-3.5) .. (.5,.5);

\node at (3,.5) {\huge{$T_2$}};
\node at (7,.5) {\huge{$T_1$}};

\node at (3, 2.5){$\alpha_{f_1}$};
\node at (6, -1){$\alpha_{f_2}$};
\node at (4,-4){\huge{$(C)$}};
\end{tikzpicture}}
\scalebox{0.7}{
\begin{tikzpicture}
\begin{knot}[
%draft mode=crossings, 
clip width = 15, 
flip crossing = 1,
%flip crossing = 2
]
\strand[black , thick] (11,0) .. controls +(1,0) and +(0,1) .. (12,-1) .. controls +(0,-2) and +(-3,-3) .. (0,0) .. controls +(0,0) and +(-1,0) .. (2,1);
\strand[black , thick] (11,1) .. controls +(1,0) and +(0,-1) .. (12,2) .. controls +(0,2) and +(-3,3) .. (0,1) .. controls +(0,0) and +(-1,0) .. (2,0);
\end{knot}
\draw[black, ultra thick] (2, -.53) -- (2,1.5) -- (4,1.5) -- (4,-.5) -- (2,-.5);
\draw[black, ultra thick] (5, .5) -- (5,1.5) -- (6,1.5) -- (6,.5) -- (5,.5);
\draw[black, ultra thick] (7, .5) -- (7,1.5) -- (8,1.5) -- (8,.5) -- (7,.5);
\draw[black, thick] (4,1) -- (5,1);
\draw[black, thick] (6,1) -- (7,1);
\draw[black, thick] (8,1) -- (9,1);
\draw[black, thick] (4,0) -- (9,0);
\draw[black, ultra thick] (9, -.53) -- (9,1.5) -- (11,1.5) -- (11,-.5) -- (9,-.5);

\draw[black, thick, dotted] (4.5,1) .. controls +(0,1.2) and +(0,1.2) .. (8.5,1);
\draw[black, thick, dotted] (8.5,1) .. controls +(0,-1.2) and +(0,-1.2) .. (4.5,1);

\draw[black, thick, dotted] (.5,.5) .. controls +(0,3) and +(0,3) .. (6.5,.5);
\draw[black, thick, dotted] (6.5,.5) .. controls +(0,-3) and +(0,-3) .. (.5,.5);

\node at (3,.5) {\huge{$T_2$}};
\node at (10,.5) {\huge{$T_1$}};
\node at (7.5,1){$T_3$};
\node at (5.5,1){$T_4$};

\node at (3, 2.3){$\alpha_{f_1}$};
\node at (7.5, 2){$\alpha_{f_2}$};
\node at (6,-4){\huge{$(A)$}};
\end{tikzpicture}}
\caption{Three potential configurations for the composition of two flypes $f_1:\lambda(D) \to \lambda(D')$ and $f_2:\lambda(D') \to \lambda(D'')$ with $c_{f_1} = c^{f_2}$. The diagrams shown are $D$ (as opposed to $D'$ or $D''$), and the domain $\alpha_{f_2}$ shown for $f_2$ is the preimage of the domain under $f_1$.}
\label{fig:orientations}
\end{figure}
\begin{lemma}
\label{lemma:commute}
If $f_1: \lambda(D_1) \to \lambda(D_2)$ and $f_2: \lambda(D_2) \to \lambda(D_3)$ are flypes such that the crossing created by $f_1$ is not the crossing removed by $f_2$, and $K$ is prime, then there exists a pair of flypes $f_2': \lambda(D_1) \to \lambda(D_2')$ and $f_1': \lambda(D_2') \to \lambda(D_3)$ such that 
\begin{enumerate}
\item $f_2 \circ f_1 = f_1' \circ f_2'$,
\item $f_2(c_{f_1}) = c_{f_1'}$, and
\item $f_1'(c_{f_2'}) = c_{f_2}$.
\end{enumerate}
Furthermore, the domains of $f_1'$ and $f_2'$ are either disjoint or nested (See Figure \ref{fig:flype}).
Informally, we will use this lemma to say that $f_1$ and $f_2$ commute, and by abuse of notation we will refer to $f_1'$ as $f_1$ and to $f_2'$ as $f_2$. 
\end{lemma}
As a further abuse of notation, given a crossing ball $c$ in $\lambda(D)$ and a flype $f:\lambda(D) \to \lambda(E)$ which does not remove $c$, we will refer to $f(c)$ as just $c$. 
\begin{remark}
\label{rmk:flype}
Note that while $f'_1(c_{f'_2}) = c_{f_2}$, the crossings created by $f_1$ and $f'_1$ may be different. However since this replacement process only reduces the domains of $f_1,f_2$, a single replacement can be done for an arbitrary composition of flypes after which commuting them does not affect which crossings they create and remove.
\end{remark}
\begin{proof}[Proof of Lemma \ref{lemma:commute}]
Again, by Lemmas \ref{lemma:unique} and \ref{lemma:exists} it is enough to prove this lemma diagrammatically on the underlying graphs. In the case where the domain for $f_1$ is contained in the domain for $f_2$, and the case where the domains for $f_1$ and $f_2$ are disjoint, this lemma is clear by defining $f_1'$ and $f_2'$ to have the same domains as $f_1$ and $f_2$ respectively. On the other hand, suppose that the domains intersect but are not nested. Then we have the configuration shown in Figure \ref{fig:commute}. In this case, define $f_1'$ as the flype with domain $T_1$ and define $f_2'$ as the flype with domain $T_3$, and the result is again clear.  
\end{proof}
\begin{figure}
\scalebox{0.7}{
\begin{tikzpicture}
\begin{knot}[
%draft mode=crossings, 
clip width = 15, 
flip crossing = 1,
%flip crossing = 2
]
\strand[black , thick] (14.5,0) .. controls +(1,0) and +(0,1) .. (15.5,-1) .. controls +(0,-2) and +(-3,-3) .. (0,0) .. controls +(0,0) and +(-1,0) .. (2,1);
\strand[black , thick] (14.5,1) .. controls +(1,0) and +(0,-1) .. (15.5,2) .. controls +(0,2) and +(-3,3) .. (0,1) .. controls +(0,0) and +(-1,0) .. (2,0);
\strand[black , thick] (11,1) .. controls +(1,0) and +(-1,-0) .. (12.5,0);
\strand[black , thick] (11,0) .. controls +(1,0) and +(-1,0) .. (12.5,1);
\end{knot}
\draw[black, ultra thick] (2, -.53) -- (2,1.5) -- (4,1.5) -- (4,-.5) -- (2,-.5);
\draw[black, ultra thick] (5.5, -.53) -- (5.5,1.5) -- (7.5,1.5) -- (7.5,-.5) -- (5.5,-.5);
\draw[black, ultra thick] (12.5, -.53) -- (12.5,1.5) -- (14.5,1.5) -- (14.5,-.5) -- (12.5,-.5);
\draw[black, thick] (4,1) -- (5.5,1);
\draw[black, thick] (7.5,1) -- (9,1);
\draw[black, thick] (4,0) -- (5.5,0);
\draw[black, thick] (7.5,0) -- (9,0);
\draw[black, ultra thick] (9, -.53) -- (9,1.5) -- (11,1.5) -- (11,-.5) -- (9,-.5);

\draw[black, thick, dotted] (.5,.5) .. controls +(0,3) and +(0,3) .. (8.25,.5);
\draw[black, thick, dotted] (8.25,.5) .. controls +(0,-3) and +(0,-3) .. (.5,.5);

\draw[black, thick, dotted] (4.75,.5) .. controls +(0,3) and +(0,3) .. (11.75,.5);
\draw[black, thick, dotted] (11.75,.5) .. controls +(0,-3) and +(0,-3) .. (4.75,.5);

\node at (3,.5) {\huge{$T_3$}};
\node at (10,.5) {\huge{$T_1$}};
\node at (6.5,.5) {\huge{$T_2$}};
\node at (13.5,.5) {\huge{$T_4$}};

\node at (3, 2.3){$\alpha_{f_1}$};
\node at (10.5, 2){$\alpha_{f_2}$};
\end{tikzpicture}}
\caption{A possible configuration for two flypes $f_1: \lambda(D) \to \lambda(D')$ and $f_2:\lambda(D') \to \lambda(D)$ which overlap. The shown diagram is $D$, $\alpha_{f_1}$ is the domain for $f_1$, and the domain $\alpha_{f_2}$ shown is the preimage under $f_1$ of the domain for $f_2$.}
\label{fig:commute}
\end{figure}
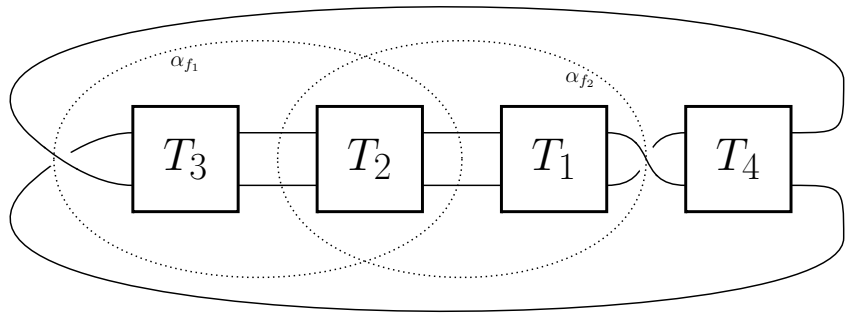

\section{Odd Prime Actions \label{sec:main}}
In this section we will apply the structure of the category of flypes developed in Section \ref{sec:flype} to periodic actions using the following theorem of Menasco and Thistlethwaite.
\begin{theorem}\cite[Main Theorem]{MT}
\label{thm:mt}
For any reduced alternating diagram $D$ for $K$ and realization $\lambda(D)$, any homeomorphism of pairs $(S^3, K) \cong (S^3, K)$ is isotopic through maps of pairs to an isomorphism of realized diagrams which is equal to a composition of flypes from $\lambda(D)$ to $\lambda(D)$.
\end{theorem}

With this theorem in hand, let $K \subset S^3$ be an alternating prime knot with a group action $\Z/p$ for $p$ an odd prime and with a reduced alternating diagram $D$. Let $\tau: S^3 \to S^3$ be a generator for the action so that $\tau^p$ is the identity map. Then by Theorem \ref{thm:mt}, $\tau$ is isotopic to a composition of flypes $\tau \cong f:= f_n \circ f_{n-1} \circ \dots \circ f_1: \lambda(D) \to \lambda(D)$, so that $\tau^p = \mbox{identity} \cong f^p$. In order to keep track of these flypes, we will refer to the $i$th iteration of $f_k$ as $f_{k,i}$ so that $f^p = f_{n,p} \circ f_{n-1,p} \circ \dots \circ f_{2,1} \circ f_{1,1}$. Note that since $f^p$ is an isomorphism $\lambda(D) \to \lambda(D)$, it takes crossing balls to crossing balls and hence induces a permutation $\sigma_{f^p}$ on the set of crossing balls $\{c_i\}$ for $\lambda(D)$. To proceed we need the following lemma.
\begin{lemma}
\label{lemma:permutation}
If $K$ is an alternating knot which is not a torus knot, then as defined above, $\sigma_{f^p}$ is the identity permutation.
\end{lemma}
\begin{proof}
To begin, we claim that $f^p$ is isotopic through maps of pairs $(S^3,K)$ preserving the crossing balls to a finite order map $g$. Clearly the permutation of crossing balls has finite order. We can then perform an homotopy on the edges of the diagram $D$ to get a map isotopic to $f$ but which has finite order when restricted to $K$. This homotopy can then be extended to the projection sphere $S^2$ since its isotopy class relative to $K$ is fixed by $f$, and from there to a map $g$ on all of $S^3$. 

But now, since $f^p$ is isotopic to the identity so is $g$, and so we have a finite order map which is isotopic to the identity. Now if $K$ is an alternating knot, it is not a satellite knot \cite{M}, and so it is either a torus knot or a hyperbolic knot. By assumption $K$ is not a torus knot, so $K$ must be hyperbolic. But by Mostow rigidity, any finite order map on a hyperbolic knot which is isotopic through maps of pairs to the identity is the identity map. Hence $\sigma_{f^p}$ is isotopic through maps preserving the crossing balls to the identity map, and hence $\sigma_{f^p}$ is the identity permutation.
\end{proof}

Now, $(f_{n} \circ \dots \circ f_1)^p$ induces the identity map on each crossing ball of $\lambda(D)$, so each crossing ball created by a flype must later be destroyed by another flype, or else remain in the final diagram. We may then separate the set of flypes into orbits $\{f_{r_1,s_1} \dots f_{r_j,s_j}\}$ such that the crossing ball created by $f_{r_i,s_i}$ is destroyed by $f_{r_{i+1},s_{i+1}}$. That is, $c_{f_{r_i,s_i}} = c^{f_{r_{i+1},s_{i+1}}}$.
\begin{lemma}
\label{lemma:orbit}
There is a reduced alternating diagram $D'$ and a choice of flypes $f_1, \dots f_{m}$ such that $\tau \cong f_{m}' \circ \dots \circ f_1' : \lambda(D') \to \lambda(D')$ and the orbit of the flype $f_{i,1}$ is $\{f_{i,1},f_{i,2}, \dots, f_{i,p}\}$.
\end{lemma}
\begin{proof}
In order to keep track of our diagram, we will refer to $f_n\circ \dots f_1: \lambda(D) \to \lambda(D)$ as $\tau_D$. Additionally, by Remark \ref{rmk:flype}, we may once and for all make a replacement of flypes so that commuting them will not change which crossings they create and remove.

 Now consider a crossing ball $c_{f_{i,j}}$ created by a flype $f_{i,j}$ which does not survive to the final diagram, and let the flype that destroys $c_{f_{i,j}}$ be $f_{k,l}$ with $k \neq i$. 

First, suppose $f_{k,l}$ and $f_{i,j}$ occur in the same iteration of $\tau$ so that $l = j$. Then we can simply commute the flypes in $\tau_D$ by Lemma \ref{lemma:commute} until $f_{i,j}$ and $f_{k,j} = f_{k,l}$ are adjacent, and then combine them into a single flype by Lemma \ref{lemma:combine}. 

Second, suppose that $l = j+1$ so that $f_{k,l}$ occurs in the iteration of $\tau_D$ directly succeeding that of $f_{i,j}$. Then by Lemma \ref{lemma:commute} we can commute $f_k$ in $\tau_D$ until $f_k$ comes just before $f_j$. That is, $\tau_D \cong f_n \circ \dots \circ f_j \circ f_k \circ f_{j-1} \circ \dots \circ f_1$. In this case, consider the realized diagram $\lambda(D')$ obtained by applying $f_k \circ f_{j-1} \circ \dots \circ f_1$ to $\lambda(D)$. We then have $\tau_{D'} =  f_k \circ f_{j-1} \circ \dots \circ f_1 \circ f_n \circ \dots \circ f_j$, so that $f_{i,j}$ and $f_{k,l}$ appear in the same iteration of $\tau_{D'}$ and we can apply the argument above to combine them. 

Finally, if $l > j+1$, then we can iterate the above change of diagram until $f_{i,j}$ and $f_{k,l}$ appear in the same iteration of $\tau_{D''}$ on some diagram $D''$ and then combine them. 

The only flypes remaining will now have orbits only containing their own later iterations, and since $p$ is prime, the orbits are exactly as stated.
\end{proof}

Using this lemma, we now reduce to the case that $\tau$ is given by a single flype and prove Theorem \ref{thm:single}. 
\begin{proof}[Proof of Theorem \ref{thm:single}]
To begin, reduce to the case that $\tau_D$ has flypes with orbits as described in Lemma \ref{lemma:orbit}. That is, the crossing ball created by each flype is only removed by a later iteration of the same flype. Now focus on a particular orbit, that of the flype $f_i$. We know that $c_{f_{i,1}} = c^{f_{i,j}}$ for some $j$, and from the proof of Lemma \ref{lemma:combine} there are two choices for the orientation of the composition: either $f_{i,j}$ continues in the direction of $f_{i,1}$ as in configuration (B) in Figure \ref{fig:orientations}, or else it flips back in the opposite direction of $f_{i,1}$ as in configuration (C) in Figure \ref{fig:orientations}. However, if it is the opposite direction, then since both of these flypes are $f_1$ they have the same tangle as their domains and so $f_{i,j}$ will exactly undo $f_{i,1}$. In other words, the orbit will consist of exactly the pair $\{f_{i,1}, f_{i,j}\}$. However, since $p \neq 2$ this is impossible, so $f_{i,j}$ must continue in the same direction as $f_{i,1}$. Hence each orbit of flypes will form a loop, and have the effect of rotating $K$ around an axis contained in the projection sphere.

Now consider the orbits for $f_1$ and $f_2$. Note that the replacement via Lemma \ref{lemma:commute} of $f_1$ and $f_2$ with $f_1'$ and $f_2'$ only reduces the domains of $f_1$ and $f_2$, and hence we may make a replacement which ensures that $f_{1,i}$ and $f_{2,j}$ commute for all $i,j$. In particular $c_{f_{1,i}}$ is not contained in the domain of any $f_{2,j}$. However, all crossings except for the $\{c_{f_{2,j}}\}$ must be contained in the domain of some $f_{1,i}$ since the collection of these flypes forms a loop which rotates the entire knot. This is a contradiction, so that in fact $\tau_D$ can be written as a single flype.
\end{proof}
We now return to the case that $K$ is $p$-periodic (as opposed to free periodic). Then iterating a single flype cannot rotate the knot around an axis in the diagram as well as an axis perpendicular to the diagram since that would be a free period and so we obtain a proof of Corollary \ref{cor:alt}.
\begin{proof}[Proof of Corollary \ref{cor:alt}]
If $K$ is a torus knot, then it is $T(2,p)$ which has a unique reduced alternating diagram, and it is $p$ periodic. If $K$ is a hyperbolic knot, then by Theorem \ref{thm:single} the periodic action is generated by a single flype $f$ on some reduced alternating diagram, and by Lemma \ref{lemma:permutation} the $p$th power of this flype is the identity map on the crossings. Then we can cut $K$ into the $p$ domains of the $f_i$, and describe $\tau$ as rotation around an axis permuting these domains, plus a rotation around an axis in the plane of the diagram. However, this describes a free periodic action unless one of these rotations is trivial. Note that the rotation around an axis in the plane of the diagram needs to be either 2-periodic or trivial, and since $p$ is odd, it is trivial. Hence $f$ is a trivial flype, and just an isomorphism of realized diagrams. In particular, $\tau$ is just a rotational symmetry of the diagram, as desired. 
\end{proof}
On the other hand, when $K$ has a $\Z/p$-action which is a free period, we get Corollary \ref{cor:free}.
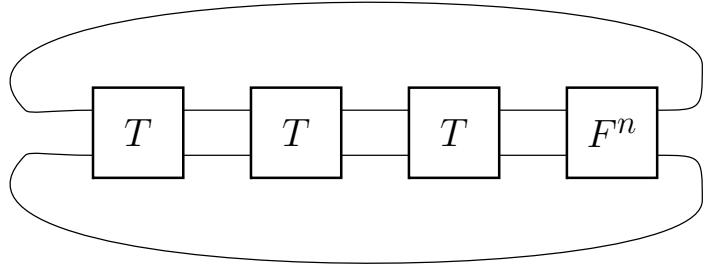
\begin{figure}
\scalebox{0.6}{
\begin{tikzpicture}
\begin{knot}[
%draft mode=crossings, 
clip width = 15, 
flip crossing = 1,
%flip crossing = 2
]
\strand[black , thick] (14.5,0) .. controls +(1,0) and +(0,1) .. (15.5,-1) .. controls +(0,-2) and +(-3,-3) .. (.5,0) .. controls +(.1,.1) and +(-1,0) .. (2,0);
\strand[black , thick] (14.5,1) .. controls +(1,0) and +(0,-1) .. (15.5,2) .. controls +(0,2) and +(-3,3) .. (.5,1) .. controls +(.1,-.1) and +(-1,0) .. (2,1);
\end{knot}
\draw[black, ultra thick] (2, -.53) -- (2,1.5) -- (4,1.5) -- (4,-.5) -- (2,-.5);
\draw[black, ultra thick] (5.5, -.53) -- (5.5,1.5) -- (7.5,1.5) -- (7.5,-.5) -- (5.5,-.5);
\draw[black, ultra thick] (12.5, -.53) -- (12.5,1.5) -- (14.5,1.5) -- (14.5,-.5) -- (12.5,-.5);
\draw[black, thick] (4,1) -- (5.5,1);
\draw[black, thick] (7.5,1) -- (9,1);
\draw[black, thick] (4,0) -- (5.5,0);
\draw[black, thick] (7.5,0) -- (9,0);
\draw[black, thick] (11,1) -- (12.5,1);
\draw[black, thick] (11,0) -- (12.5,0);
\draw[black, ultra thick] (9, -.53) -- (9,1.5) -- (11,1.5) -- (11,-.5) -- (9,-.5);

\node at (3,.5) {\huge{$T$}};
\node at (10,.5) {\huge{$T$}};
\node at (6.5,.5) {\huge{$T$}};
\node at (13.5,.5) {\huge{$F^n$}};

\end{tikzpicture}}
\caption{A $(p=3,n)$ free periodic alternating diagram. $T$ is any alternating tangle, and $F^n$ is $n$ full twists.}
\label{fig:free}
\end{figure}
\begin{proof}[Proof of Corollary \ref{cor:free}]
Again by Theorem \ref{thm:single} the free periodic action is generated by a single flype $f$ on some reduced alternating diagram. If $f$ is a trivial flype, then the the action is a period of the knot $K$, so $f$ must be a non-trivial flype. As above, we can then cut $K$ into the domains of the $f_i$, and describe $\tau$ as a permutation of these domains with a twist around an axis in the plane of the diagram in between. In particular, the diagram must already be a free periodic diagram, see Figure \ref{fig:free}.
\end{proof}
Finally, we use Corollary \ref{cor:alt} to consider the quotient of an alternating periodic knot. 
\begin{corollary}
Let $K$ be an odd-periodic alternating prime knot. Then the quotient knot of $K$ is alternating. 
\end{corollary}
\begin{proof}
By Corollary \ref{cor:alt}, we can find a reduced alternating periodic diagram $D$ for $K$. Then the quotient knot is obtained by cutting a fundamental domain of the rotation from $D$ and connecting the free ends without crossings. In particular, a strand leaving an under crossing and going to an over crossing still does so.
\end{proof}
\section{Examples \label{sec:example}}
Corollaries \ref{cor:alt} and \ref{cor:free} give an elementary method to determine if an alternating prime knot $K$ has a $p$ period or a free $p$ period. Since all reduced alternating diagrams are related by flypes, you can simply list all reduced alternating diagrams for $K$ and check what symmetries they have.
\begin{example}
Consider $4_1$, the figure eight knot. See Figure \ref{fig:4}, which shows a reduced alternating diagram. It has no possible non-trivial flypes, and it is not a periodic or free periodic diagram, so $4_1$ is not $p$ periodic or free periodic for any odd prime $p$, and hence for any odd integer $p$. 
\end{example}
\begin{figure}
\scalebox{0.5}{
\includegraphics[width=0.5\textwidth]{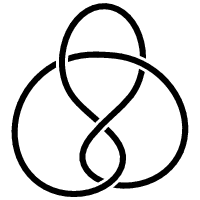}}
\caption{The unique reduced alternating diagram for $4_1$.}
\label{fig:4}
\end{figure}
More generally, the number of crossings in a reduced alternating diagram for a knot obstructs the existence of periods, although the existence of the full twists in a free periodic diagram means there is no such obstruction for free periods.
\begin{example}
Suppose $K$ is $p$ periodic for an odd prime $p$, and has a reduced alternating diagram with $n$ crossings. Then since any other reduced alternating diagram is related by a sequence of flypes which does not change the crossing number, $K$ must have a $p$ periodic diagram with $n$ crossings by Corollary \ref{cor:alt}. In particular, $n$ must be a multiple of $p$. 
\end{example}
\bibliography{bibliographyA}{}
\bibliographystyle{alpha}

\end{document}